\documentclass[11pt, reqno]{amsart}

\usepackage{amsmath,amssymb,amsthm, mathrsfs,listings}
\usepackage{cleveref}
\usepackage{color}
\usepackage[all,cmtip]{xy}

\usepackage{amssymb,latexsym}
\usepackage{amsmath,amsfonts,amsthm,amscd,amsxtra,enumerate}
\usepackage{mathtools}

\usepackage[top=3.2cm,bottom=3.2cm,left=2.5cm,right=2.5cm]{geometry}

\usepackage{epsfig}

\newtheorem{theorem}{Theorem}[section]

\newtheorem{prop}[theorem]{Proposition}

\newtheorem{cor}[theorem]{Corollary}
\theoremstyle{definition}
\newtheorem{defn}[theorem]{Definition}
\newtheorem{remark}[theorem]{Remark}
\newtheorem{example}[theorem]{Example}

\newtheorem{setup}[theorem]{Setup}

\everymath{\displaystyle}

\DeclareMathOperator{\Hom}{Hom}

\DeclareMathOperator{\rank}{rank}

\DeclareMathOperator{\hgt}{ht}

\DeclareMathOperator{\ima}{Im}

\DeclareMathOperator{\depth}{depth}
\DeclareMathOperator{\grade}{grade}

\DeclareMathOperator{\reg}{reg} 
\DeclareMathOperator{\Tor}{Tor}

\newcommand{\p}{\mathfrak{p}}

\newcommand{\m}{\mathfrak{m}}

\newcommand{\sk}{{\sf k}}

\begin{document}

\title{\textbf{Diagonal Subalgebras of Residual Intersections}}

\author[H. Ananthnarayan]{H. Ananthnarayan}
\address{Department of Mathematics, I.I.T. Bombay, Powai, Mumbai 400076.}
\email{ananth@math.iitb.ac.in}

\author[N. Kumar]{Neeraj Kumar}
\address{Department of Mathematics, I.I.T. Bombay, Powai, Mumbai 400076.}
\email{neerajk@math.iitb.ac.in, neeraj@iith.ac.in}

\author[V. Mukundan]{Vivek Mukundan}
\address{Department of Mathematics, University of Virginia, Charlottesville, VA 22904.}
\email{vm6y@virginia.edu}


\date{\today}

\begin{abstract}
Let $\sk$ be a field, $S$ be a bigraded $\sk$-algebra, and $S_\Delta$ denote the diagonal subalgebra of $S$ corresponding to $\Delta = \{ (cs,es) \; | \; s \in \mathbb{Z} \}$. It is known that the $S_\Delta$ is Koszul for $c,e \gg 0$. In this article, we find bounds for $c,e$ for $S_\Delta$ to be Koszul, when $S$ is a geometric residual intersection. Furthermore, we also study the Cohen-Macaulay property of these algebras. Finally, as an application, we look at classes of linearly presented perfect ideals of height two in a polynomial ring, show that all their powers have a linear resolution, and study the Koszul, and Cohen-Macaulay property of the diagonal subalgebras of their Rees algebras.
\end{abstract}

\maketitle

\section*{Introduction}
Let $\sk$ be a field, $c,e \in \mathbb N$, and $\Delta = \{ (cs,es) \; | \; s \in \mathbb{Z} \}$ be the $(c,e)$-diagonal of $\mathbb{Z}^2$. Given a bigraded $\sk$-algebra $S= \bigoplus_{(u,v) \in {\mathbb{Z}^2_{\geq 0}}} S_{(u,v)}$, one can associate a graded $\sk$-algebra 
$S_{\Delta}= \oplus_{s \in \mathbb{Z}} S_{(cs,es)}$, the $(c,e)$-diagonal subalgebra of $S$. In a special case, when $S=\sk[x_1,\dots,x_n,y_1,\dots,y_p]$ with $\deg x_i =(1,0)$ and $\deg y_j =(0,1)$, then $S_{\Delta}$ is the homogeneous coordinate ring of the image of the correspondence under the Segre embedding $\mathbb{P}^{n-1} \times \mathbb{P}^{p-1} \hookrightarrow \mathbb{P}^{np-1}$ (cf. \cite{SimisTrungValla}).

The motivation for the study of diagonal subalgebras comes from algebraic geometry. Many authors have studied the algebraic properties (for instance, normality, Cohen-Macaulayness, defining equations etc.) of rational surfaces obtained by blowing up projective space $\mathbb{P}^2$ at a finite set of points, for instance see \cite{GerGim, GerGimHar, Hir}. Simis, Trung and Valla introduced diagonal subalgebras to study the normality and the Cohen-Macaulayness of rational surfaces obtained by blowing up projective space $\mathbb{P}^2$ along a subvariety (see \cite{SimisTrungValla}). Conca, Herzog, Trung and Valla used diagonal subalgebras as an effective tool for the study of normality, Cohen-Macaulayness, Gorensteinness and Koszulness of embedded $(n-1)$-fold obtained by blowing up $\mathbb{P}^{n-1}$ along a subvariety (cf. \cite{ConHerTruVal}).

The notion of residual intersections was introduced and initially studied by Artin-Nagata (\cite{ArtNag}), and Huneke-Ulrich (\cite{HunUlrResInt}). Our motivation for studying geometric residual intersections is the following: Morey and Ulrich (\cite{MorUlr}) show that if $I$ is a linearly presented perfect ideal of height two in a polynomial ring, its Rees algebra $\mathcal R(I)$ is a geometric residual intersection when $I$ satisfies certain conditions (see \Cref{setup}). For such an ideal $I$, the authors in \cite{ConHerTruVal} also show that $\mathcal R(I)_\Delta$ is Cohen Macaulay for $c,e \gg 0$. 

The Koszul property of the diagonal subalgebras of the Rees algebra of ideals generated by a regular sequence has been studied in multiple articles. In particular, for an ideal $I$, generated by a homogeneous regular sequence $f_1, f_2,f_3 \in \sk[x_1,\ldots, x_n]$ of the same degree, explicit lower bounds are obtained in \cite{CavCon} and \cite{Neeraj}, such that for larger values of $c$ and $e$, $\mathcal R(I)_\Delta$ is Koszul. More generally, the authors in \cite{ConHerTruVal} also show that for given any standard bigraded $\sk$-algebra $S$, one has $S_{\Delta}$ is Koszul for $c,e \gg 0$.

Motivated by the above results, we focus on giving explicit bounds for $c,e$ for diagonal subalgebras of certain classes of geometric residual intersections to be Koszul, or Cohen-Macaulay.  More specifically, in this article, given positive integers $m \geq n$, we study an ideal $J = \langle z_1, \ldots z_m \rangle + I_n(\phi)$ in the polynomial ring $S = \sk[x_1, \ldots, x_n, y_1, \ldots, y_p]$, where $\phi$ is an $n \times m$  matrix, with entries linear in $\underline y$, such that $ \underline z = \underline x \phi$. When $\hgt(J) \geq m$, and $\hgt(I_n(\phi)) \geq m - n + 1$, $J$ is a geometric residual intersection of $\langle x_1,\dots,x_n\rangle$. For such ideals, Bruns, Kustin and Miller  construct an $S$-free resolution of $S/J$ (\cite[Theorem 3.6]{BKM1}). 


Our main tool is to show that this resolution gives a bigraded resolution of $S/J$ over $S$, when $S$ is bigraded with $\deg(x_i) = (1,0)$ and $\deg(y_j) = (0,1)$. Using the construction given in \cite{BKM1}, we first compute the bigraded shifts in \Cref{bigradedShifts}, and use this to compute the bigraded Betti numbers, $x$-regularity and $y$-regularity of $S/J$ over $S$ in \Cref{Prop:shift}. In order to understand the Cohen-Macaulay property, we use this to give a bound on the depth of $(S/J)_\Delta$ in \Cref{thm:DepthGenResInt} for all $\Delta$. Moreover, we also show that it is Koszul when $e \geq n/2$, in \Cref{thm:KoszulGeomResInt}.

We give applications of the results in \Cref{Sec:DigGeoRes} to perfect ideals $I$ of height two, which are linearly presented in a polynomial ring $R = \sk[x_1,\ldots,x_n]$, which further satisfy $\mu(I_\p)\leq \dim(R_\p)$ for every $\p\in V(I)$ with $\dim(R_\p)\leq p-1$. This property can also be reformulated as the ideal $I$ being of \emph{linear type in the punctured spectrum}. In \Cref{thm:ReesAlg}, we first show that all powers of such ideals have a linear resolution, generalising a result of R\"omer (\cite[Corollary 5.11]{Tim}). We also prove that the diagonal subalgebras of the Rees algebra of $I$ are Cohen-Macaulay when $c>(p-1)e$, and are Koszul when $e \geq n/2$. The Koszul nature of diagonal subalgebras of the Rees algebras of such ideals has not been studied before.

\Cref{sec:prel} contains the definitions, basic observations, and known properties of the main objects appearing in this paper. We end the paper with an example in \Cref{Sec:Example} to understand the construction given in \Cref{Subsec:bigradedresGeoRes} explicitly, and to apply it to understand our results.
\subsection*{Acknowledgements} We thank the anaonymous referee for  comments and suggestions which improved the accuracy of the article greatly.

\section{Preliminaries}\label{sec:prel}
\subsection*{Notations}
\begin{enumerate}[a)]
\item $c$ and $e$ denote positive integers, $\Delta = \{(ci,ei)\ \vert\ i \in \mathbb Z\}$ is the $(c,e)$-diagonal of $\mathbb Z^2$, and for a real number $\alpha$, we let $\lceil \alpha \rceil$ denote the least integer greater than or equal to $\alpha$.

\item $\sk$ denotes a field, $\underline x = x_1, \ldots x_n$, $\underline y = y_1,\ldots, y_p$ denote sets of indeterminates over $\sk$, $R_x = \sk[\underline x]$ and $R_y = \sk[\underline y]$ are polynomial rings. 

\item $R$ denotes a \emph{graded $\sk$-algebra}, i.e., $R = \oplus_{i \geq 0} R_i$ is a graded ring with $R_0 = \sk$. 
Let $R_+ = \oplus_{i \geq 1} R_i$ be the unique homogenous maximal ideal of $R$. We say that $R$ is \emph{standard graded} if $R_+$ is generated by $R_1$. 

\item Given a graded $R$-module $M$ and $j \in \mathbb Z$, the \emph{shifted module} $M(j)$ is the graded $R$-module with $i$th graded component $M(j)_i = M_{i+j}$. In particular, $R(-j)$ is the graded free $R$-module of rank one, with the generator in degree $j$.

\item $S = \bigoplus_{(i,j) \in \mathbb Z^2_{\geq 0}} S_{(i,j)}$ denotes a bigraded $\sk$-algebra, i.e., $S$ is bigraded and $S_{(0,0)}= \sk$. Just as in the graded case, for integers $a,b$, the bigraded free $S$-module of rank one, generated in bidegree $(a,b)$, is denoted $S(-a,-b)$. 
\end{enumerate}

\subsection*{Eagon-Northcott Complex}

Let $S=R_x \otimes_\sk R_y = \sk[\underline x, \underline y]$ be a polynomial ring over $\sk$. Let $m \geq n$, and $\phi$ be an $n \times m$ matrix with entries in $R_y$. Define $z_1, \ldots, z_m \in S$ by 
\begin{align*}
\begin{bmatrix}
z_1 &z_2 &\ldots &z_m
\end{bmatrix}=\begin{bmatrix}
x_1 & x_2 &\ldots& x_n
\end{bmatrix}\cdot \phi.
\end{align*}

Consider the Koszul complex $\mathbb{K}_{\bullet}(\underline{z}; S)$ on $\underline z = z_1, z_2, \ldots, z_m$:
\begin{align*}
0 \longrightarrow S(-m)^{\oplus{ m \choose m}}  \xrightarrow{\phi^m} S(-m +1)^{\oplus {m \choose m-1}} \xrightarrow{\phi^{m-1}} \cdots  \longrightarrow S(-2)^{\oplus{m \choose 2}} \xrightarrow{\phi^2}   S(-1)^{\oplus{m \choose 1}} \xrightarrow{\phi^1} S\rightarrow 0.
\end{align*}
The $d$-th graded component, with respect to the $x$-grading, of the above complex gives us a complex of free $R_y$-modules:
\begin{align}\label{eqn:Koszul complex on z_i-d}
\mathbb{K}_{\bullet}(\underline{z}; S)_d: \hfill 0 \longrightarrow S_{d-m} \xrightarrow{\phi_d^m} (S_{d-m+1})^{\oplus {m \choose m-1}} \longrightarrow \cdots  \longrightarrow (S_{d-2})^{\oplus {m \choose 2}} \xrightarrow{\phi_d^2} (S_{d-1})^{\oplus {m \choose 1}} \xrightarrow{\phi_d^1} S_d \; .
\end{align}
Applying $\Hom_{R_y}(\_\!\_, R_y) = (\_\!\_)^\ast$ to \Cref{eqn:Koszul complex on z_i-d}, we get a new complex of $R_y$-modules: 
\begin{align}\label{eqn:Koszul complex dual on z_i-d}
\mathbb{K}_{\bullet}(\underline{z}; S)_d^{\ast}: \; 0  \rightarrow  S_d^{\ast}  \xrightarrow{{(\phi_d^1)^*}}  (S_{d-1}^{\ast})^{\oplus {m \choose 1}}   \longrightarrow \cdots  \xrightarrow{{(\phi_d^{m-1})^*}}  (S_{d-m+1}^{\ast})^{\oplus {m \choose m-1}} \xrightarrow{(\phi_d^m)^*} S_{d-m}^{\ast}  \rightarrow 0.
\end{align}

\begin{defn}{[Eagon-Northcott Complex]}\label{EagonNorthcott}
With the notation as above, the Eagon-Northcott complex of the matrix $\phi$ is given by 
 \begin{align*}
\mathbb{EN}: 0  \rightarrow  S_{m-n}^{\ast}  \xrightarrow{(\phi_{m-n}^1)^*}  (S_{m-n-1}^{\ast})^{\oplus {m \choose 1}}  \longrightarrow  \cdots  \longrightarrow  (S_{1}^{\ast})^{\oplus {m \choose m-n-1}} \xrightarrow{(\phi_{m-n}^{m-n})^*} (S_{0}^{\ast} )^{m\choose m-n}\xrightarrow{\epsilon} R_y \rightarrow 0
 \end{align*}
 where the map $\epsilon$ is given by the $1\times {m\choose m-n}$ row matrix whose entries are the $(n\times n)$ - minors of the matrix $\phi$. 
 \end{defn}
 
\begin{remark}{\rm
Notice that the first $m-n$ components of $\mathbb{EN}$ are nothing but $\mathbb{K}_{\bullet}(\underline{z}; S)_{m-n}^{\ast}$. The complex $\mathbb{EN}$ is exact when $\grade \left(I_n(\phi)\right)=m-n+1$ (e.g., \cite[Theorem A2.10]{EisenbudCommAlg}). Observe that the homology at the zeroth stage of the above complex is $R_y/(\ima (\epsilon))=R_y/I_n(\phi)$. Hence, if $\grade I_n(\phi)=m-n+1$, we see that $\mathbb{EN}$ gives a minimal free resolution of $R_y/I_n(\phi)$ over $R_y$.
}\end{remark}

\subsection*{Segre Product, Veronese and diagonal subalgebras}

\begin{defn} Let $c,e \in \mathbb N$, $S$ a bigraded $\sk$-algebra, $W$ a bigraded $S$-module, $A$ and $B$ be graded $\sk$-algebras, $M$ and $N$ graded $A$ and $B$ modules respectively. 
\begin{enumerate}[a)]
\item The \emph{Segre product} of $A$ and $B$ is defined as $A \underline{\otimes}_{\sk} B   = \bigoplus_{i \geq 0} ( A_i \otimes B_i) $. The Segre product of $M$ and $N$ is the $A \underline{\otimes}_{\sk} B$-module $M\underline{\otimes}_{\sk} N = \bigoplus_{i \in \mathbb{Z}}  (M_i \otimes N_i) $.

\item The $c$-th \emph{Veronese subring} of $A$ is is the graded $\sk$-algebra $A^{(c)} = \bigoplus_{i \geq 0} A_{ci}$. 

\item The $(c,e)$-\emph{diagonal subalgebra} of $S$ is the graded $\sk$-algebra $S_{\Delta} = \bigoplus_{i \geq 0} S_{(ci,ei)}$, and the $(c,e)$-\emph{diagonal module} corresponding to $W$ is the graded $S_{\Delta}$-module $W_\Delta = \bigoplus_{i \geq 0} W_{(ci,ei)}$. 
\end{enumerate}
\end{defn}

\begin{remark}\label{remarkBigraded}
{\rm Let the notation be as in the above definition.

\begin{enumerate}[a)]
\item $(\_\!\_)_\Delta$ is an exact functor from the category of bigraded $S$-modules to the category of graded $S_\Delta$-modules. However, for  integers $a, b$, note that $S(-a,-b)_{\Delta} = \bigoplus_{i \in\mathbb{Z}} S_{(-a+ci, -b + ei)}$ need not be a free $S_\Delta$-module.
\item For a bihomogeneous ideal $J$ in $S$, $(S/J)_{\Delta}$ has a natural graded $\sk$-algebra structure by definition.
\item The $\sk$-algebra $A \otimes_{\sk} B$ is naturally bigraded and the Segre product $A \underline{\otimes}_{\sk} B$ is the $(1,1)$-diagonal subalgebra of $A \otimes_{\sk} B$. More generally, the $(c,e)$-diagonal subalgebra of $A \otimes_{\sk} B$ is the Segre product of the Veronese subalgebras $A^{(c)}$ and $B^{(e)}$.
\item \label{rem:S(-a,-b)deltaCM}  Let $S = \sk[x_1,\ldots,x_n,y_1\ldots,y_p]$ with bigrading $\deg(x_i) = (1,0)$ and $\deg(y_j) = (0,1)$ for all $i$ and $j$. By \cite[Lemmas $3.1$, $3.3$]{ConHerTruVal} we have $\dim \left( S(-a,-b)_{\Delta} \right) = p + n -1$. Finally, \cite[Lemma 3.10]{ConHerTruVal} shows that $S(-a,-b)_{\Delta}$ is Cohen-Macaulay 
if and only if the following two conditions are satisfied: \\(i) $\left\lfloor\frac{a-n}{c} \right\rfloor < \frac{b}{e}$ and (ii) $ \left\lfloor\frac{b-p}{e} \right\rfloor <\frac{a}{c}$.
\end{enumerate}
}
\end{remark}


\subsection*{Residual Intersections and Rees Algebras}
Let $R$ be any Cohen-Macaulay local ring, $I \subset R$ be an ideal and $\underline z = z_1,\dots,z_m \in I$ with $\langle \underline z \rangle \neq I$.
\begin{defn}
	We say $J=\langle \underline z \rangle:I$ is an $m$-residual intersection of $I$ if $\hgt (J)\geq m\geq \hgt (I)$. Furthermore, if $I_{\p}=\langle \underline z \rangle_{\p}$ for all $\p\in V(I)$ with $\hgt (\p)\leq m$, then $J$ is called an $m$-geometric residual intersection of $I$.
\end{defn}

\begin{example}
{\rm
	Consider the ring $R=k[x_1,x_2,y_1,\ldots,y_4], I=(x_1,x_2)$ and $[z_1~z_2]=[x_1~x_2]\cdot \phi$ where $\phi=\begin{bmatrix}
	y_1 & y_3\\
	y_2 & y_4
	\end{bmatrix}$. Then $\langle z_1,z_2\rangle :I$ is a geometric residual intersection of $I$ \cite[Theorem 3.3(i)]{HunUlrResInt}.
}
\end{example}

\begin{defn}\label{ReesAlgebraDefn} Let $I = \langle f_1, \ldots, f_p \rangle$ be a homogeneous ideal in $R_x = \sk[x_1,\dots,x_n]$. We say that $\Phi$ is a presentation matrix of $I$ if $(\sk[\underline x])^m\overset{\Phi}\longrightarrow(\sk[\underline x])^p\rightarrow I\rightarrow 0$ is exact.  The \emph{Rees algebra} of $I$ is the subalgebra of $\sk[\underline x,t]$ defined as $\mathcal{R}(I) = \sk[\underline x, f_1t, \ldots, f_pt]$. 
\end{defn}

\begin{remark}\label{remarkReesAlgebras}
{\rm Let $S = R_x \otimes_\sk R_y = \sk[x_1,\dots,x_n, y_1,\ldots, y_p]$.
\begin{enumerate}[a)]

\item The ring $S$ maps naturally onto $\mathcal{R}(I)$ by $x_i \mapsto x_i$ and $y_j \mapsto f_jt$. The kernel $\mathcal{K}$ of this map is called the defining ideal of the Rees algebra $\mathcal{R}(I)$.

\item Let $[z_1~\cdots ~z_m]=[y_1~\cdots y_p]\cdot\Phi$, where $\Phi$ is a $p\times m$ presentation matrix of $I$. Then we have $\langle z_1,\dots,z_m\rangle \subset \mathcal{K}$. If  $\mathcal{K}=\langle z_1,\dots,z_m\rangle$, then we say that $I$ is an ideal of \textit{linear type}.

\item \label{rem:dimofdiagonalofreesalgebra} Assume that $I^e$ is generated by forms of degree atmost $c-1$. By \cite[Lemma $1.2$ and $1.3$]{ConHerTruVal}, we get that $\dim \left( \mathcal{R}(I)_{\Delta} \right) = n$.
\item \label{rem:dimofdiagonalofreesalgebraForc=e=1} If $\Delta=(c,e)=(1,1)$, then by \cite[Section 3, Proposition 2.3]{SimisTrungValla}, we have $\dim \left( \mathcal{R}(I)_{\Delta} \right) = n$.
\end{enumerate}
}\end{remark}

\subsection*{Regularity and the Koszul property}

\begin{defn}
Let $R$ be a standard graded $\sk$-algebra, and $M$ be a finitely generated graded $R$-module with $(i,j)$-th graded Betti number $\beta_{ij}^R(M)=\dim_{\sk} \left( \Tor_i^R(M,\sk)_j \right)$.
\begin{enumerate}[a)]
\item The \emph{Castelnuovo-Mumford regularity} of $M$ over $R$ is
$\reg_R(M)= \sup_{i \geq 0} \left\{   j - i  \;| \;  \beta_{ij}^R(M) \neq 0  \right\}$.
\item Let $I$ be a homogeneous ideal in $R$ generated in degree $d$. Then $I$ has a linear resolution if $\reg_R(I) = d$, i.e., if for all $i$, $\beta_{ij}^R(I) = 0$ for $j \neq i+ d$.
\item We say that $R$ is a \emph{Koszul algebra} if $\reg_R(\sk) = 0$, i.e., if for all $i$, we have $\beta_{ij}^R(\sk) = 0$ for $j \neq i$.
\end{enumerate}
\end{defn}

\begin{defn}
Let $S = \sk[\underline x, \underline y]$ with $\deg (x_i)=(1,0)$ and $\deg (y_j) = (0,1)$ be a bigraded polynomial ring, and $J \subset S$ a bigraded ideal. If $\beta_{i,(a,b)}^S(S/J)=\dim_{\sk} \left( \Tor_i^S(S/J,\sk)_{(a,b)} \right)$ denote the bigraded Betti numbers of $S/J$, we define the $x$-regularity and $y$-regularity of $S/J$ to be: 
\[
\begin{split}
\reg_{x}^{S}(S/J)&= \sup\{ a -i ~ |~ \beta_{i,(a,b)}^S(S/J) \neq 0  \text{ for some $i,b \in \mathbb{Z} $}  \},\\
\reg_{y}^{S}(S/J)&= \sup\{ b-i ~ |~  \beta_{i,(a,b)}^S(S/J) \neq 0  \text{ for some $i,a \in \mathbb{Z} $}    \}.\
\end{split}
\]
\end{defn}

\begin{remark}\label{rem:regularity}
{\rm Let $R$ be a standard graded $\sk$-algebra, and $S = R_x\otimes_{\sk} R_y$ be naturally bigraded by setting $\deg (x_i)=(1,0)$ and $\deg (y_j) = (0,1)$.

\begin{enumerate}[a)]
\item\label{reglemma} Given an exact sequence $0 \longrightarrow N_m \longrightarrow \cdots \longrightarrow N_2 \longrightarrow N_1 \longrightarrow N_0 \longrightarrow M \longrightarrow 0$  of graded $R$-modules, repeated application of the depth lemma (e.g., \cite[Corollary 18.16]{EisenbudCommAlg} ) and the regularity lemma (e.g., \cite[Corollary $20.19$]{EisenbudCommAlg}) respectively yield the following:\\
$\depth_R(M) \geq \min\{\depth_R(N_i) - i\ \vert\ 0 \leq i \leq m\}$ and $\reg_R(M) \leq \max\{\reg_R(N_i) - i\ \vert\ 0 \leq i \leq m\}$.

\item\label{thm:Froberg}\textnormal{\cite[Theorem 4]{BacFro}}
The Segre product of two Koszul algebras are Koszul. Moreover, Veronese subrings of Koszul algebras are Koszul.
In particular, $S_{\bigtriangleup}= (R_x)^{(c)} \underline{\otimes} (R_y)^{(e)}$ is Koszul.

\item \cite[Lemma $6.5$]{ConHerTruVal}\\ $\reg_{\left(R_x^{(c)} \underline{\otimes} R_y^{(e)} \right)} \left(R_x(-a)^{(c)}  \underline{\otimes } \; R_y(-b)^{(e)}\right)  
=  \max \left\{   \reg_{R_x^{(c)}} \left(R_x(-a)^{(c)}\right)   ,  \reg_{R_y^{(e)}} \left(R_y(-b)^{(e)}\right)  \right\}$.

\item\label{regofshifteddiagonal} By \cite[Theorem $2.1$]{ABH}, we have $\reg_{R_x^{(c)}} \left(R_x(-a)^{(c)}\right) = \left\lceil \frac{a}{c}\right\rceil$.\\ Hence by (c), we get $\reg_{S_{\Delta}}\left(S(-a,-b)_{\Delta}\right) = \max\left\{ \left\lceil \frac{a}{c} \right\rceil, \left\lceil \frac{b}{e} \right\rceil \right\}$.

\item\label{thm:koszulcriterion}\textnormal{\cite[Lemma $6.6$]{ConHerTruVal}}
Let $I$ be a homogeneous ideal in $R$ with $\reg_R (R/I) \leq 1$. If $R$ is Koszul, then so is $R/I$.

\item \cite[Theorem $5.3$]{Tim}\label{thm:Romer} Let $I = \langle f_1, \ldots, f_p \rangle \subset R_x$ be a graded ideal generated in degree $d$. Write the Rees algebra of $I$ as $\mathcal{R}(I)\cong S/\mathcal{K}$ (as in \Cref{remarkReesAlgebras} (a)). Then $ \reg_{R_x}(I^s) \leq s d +    \reg_{x}^{S}\left( \mathcal{R}(I) \right)$. In particular, if $\reg_x(\mathcal R(I)) = 0$, then $I^s$ has a linear resolution for all $s$.
\end{enumerate}
}
\end{remark}


\section{Geometric Residual Intersections and their Diagonal Subalgebras}\label{Sec:DigGeoRes}

Let $S=R_x[y_1,\ldots,y_p]$ be a polynomial ring over $R_x$, and $\m = \langle \underline x\rangle$. Define a bigrading on $S$ by setting $\deg(x_i)=(1,0),\deg(y_j)=(0,1)$ for $1\leq i\leq n$ and $1\leq j\leq p$.

As in the definition of the Eagon-Northcott complex, let $m \geq n$, and $\phi$ be an $n \times m$ matrix with linear entries in $R_y$, and let $\underline z = z_1, \ldots, z_m \in S$ be given by 
\begin{align*}
\begin{bmatrix}
z_1 &z_2 &\ldots &z_m
\end{bmatrix}=\begin{bmatrix}
x_1 & x_2 &\ldots& x_n
\end{bmatrix}\cdot \phi.
\end{align*}

We study the bigraded $S$-ideals of the form $J=\langle \underline z\rangle+I_n(\phi)$ where $\grade (I_n(\phi))\geq m-n+1$. Suppose $\hgt(J)\geq m$. Since $J\subseteq \langle \underline{z} \rangle:\m$ (e.g., by Cramer's Rule), we see that $\langle \underline z\rangle:\m$ is an $m$-residual intersection of $\m$. If we further assume that $\langle\underline z\rangle:\m$ is a geometric $m$-residual intersection of $\m$, then it is shown in the proof of \cite[Theorem 4.8]{BKM1} that equality holds, i.e., $J = \langle \underline z\rangle:\m$. 

With $J$ as above, the authors in \cite{BKM1} construct an $S$-free resolution of $S/J$. We show that this is a bigraded resolution, and begin with a brief review of the construction in \cite{BKM1} below. 

\subsection*{Bigraded resolution of $S/J$}\label{Subsec:bigradedresGeoRes} With notation as above, let $\mathbb{K}_{\bullet}(\underline{x}; S)$ be the Koszul complex on the sequence $x_1,\dots,x_n$.  As in the construction of the Eagon-Northcott complex, let $\mathbb{K}_{\bullet}(\underline{x}; S)_d$ be the corresponding $d$-th graded component with differentials $\psi^i_d$.  Set $K_a^b=\ker ((\psi_{a+n-b}^{n-b+1})^*)$.

\begin{remark}\label{rankOfKG} In \cite{BKM1}, the authors have denoted the kernel $K_a^b$ as $K_a^b(G)$ and the map $(\psi_{a+n-b}^{n-b+1})^*$ by $\eta_{a}^{b}$, hence the need for the notation of $K_a^b$ presented above. In (cf. \cite[Proposition 1.13]{BKM1}), the authors also show that each $K_a^b$ is a free $R_y$-module with $\rank_{R_y}(K_a^b) =  \binom{n+a-1-b}{a}\binom{n+a}{b}$ .
\end{remark}

Consider the following bi-complex $\mathbb{B}_{\bullet}$ of free $S$-modules:

 \begin{align}\label{bicomplex}
 \xymatrix{
 	0\ar[r] & K_{m-n}^{n-1}\otimes S \ar[r] \ar[d]& \cdots \ar[r]& K_1^{n-1}\otimes S^{m\choose n+1}\ar[r]\ar[d]&K_0^{n-1}\otimes S^{m\choose n}\ar[r]\ar[d] &S^{m\choose n-1}\ar[d]^{\phi^{n-1}}\\
 	0\ar[r] & K_{m-n}^{n-2}\otimes S\ar[r] \ar[d]& \cdots \ar[r]&K_1^{n-2}\otimes S^{m\choose n+1}\ar[r]\ar[d] &K_0^{n-2}\otimes S^{m\choose n}\ar[r]\ar[d] &S^{m\choose n-2}\ar[d]^{\phi^{n-2}}\\
 	& \vdots &&\vdots &\vdots& \vdots\\
 	0\ar[r] & K_{m-n}^{1}\otimes S\ar[r] \ar[d]& \cdots \ar[r]&K_1^{1}\otimes S^{m\choose n+1}\ar[r]\ar[d] &K_0^{1}\otimes S^{m\choose n}\ar[r]\ar[d] &S^{m\choose 1}\ar[d]^{\phi^{1}}\\
 	0\ar[r] & K_{m-n}^{0}\otimes S\ar[r] & \cdots \ar[r]&K_1^{0}\otimes S^{m\choose n+1}\ar[r] &K_0^{0}\otimes S^{m\choose n}\ar[r] &S.
 }
 \end{align}

Under the given conditions, it was proved in \cite[Theorem 3.6]{BKM1} that the total complex $\mathbb{T}_{\bullet}$ of the bicomplex  $\mathbb{B}_{\bullet}$ is a free resolution of $S/J$. Thus, we see that 
\begin{align}\label{total complex}
\mathbb{T}_{\bullet}:\quad 0\rightarrow P_m\rightarrow P_{m-1}\rightarrow\cdots\rightarrow P_{n}\rightarrow Q_{n-1}\oplus P_{n-1}\rightarrow\cdots\rightarrow Q_1\oplus P_1\rightarrow Q_0\rightarrow S/J \rightarrow 0
\end{align}
is an $S$-free resolution of $S/J$, with $Q_i=S^{m\choose i}$ for $0 \leq i \leq n-1$, and $P_i=\bigoplus K_a^b\otimes S^{m\choose n+a}$, $1 \leq i \leq m$, where the direct sum is over all $(a,b)$ such that 
\begin{align}\label{limitsForab}
a+b=i-1, && 0\leq a\leq m-n,&& 0\leq b\leq n-1. 	
\end{align}  

We show that this is a bigraded resolution of $S/J$ over $S$ by computing the bigraded shifts in $\mathbb{T}_\bullet$. In order to do this, we first compute the bigraded shifts in the bicomplex $\mathbb {B}_\bullet$.
\begin{remark}\label{bigradedShifts}
The bigraded shifts appearing in $\mathbb {B}_\bullet$ are as follows:
\begin{enumerate}[a)]
	\item The column on the right of the above diagram is the truncation of the Koszul complex $\mathbb{K}_{\bullet}(\underline{z}; S)$. Thus the maps on the last column are the differentials $\phi^i$, which are bigraded maps of bidegree $(1,1)$.
	\item Using \cite[Lemma 2.5]{BKM1} the maps appearing in the other columns are induced by $\mathbb{K}_\bullet(\underline{x};S)$ and hence are bigraded maps of bidegree $(1,0)$. 
	\item The horizontal map between the last two columns is the augmentation map $\epsilon^b$ as explained in \cite[Corollary 2.7]{BKM1}. Thus this augmentation map $K_0^b\otimes S^{m\choose n}\rightarrow S^{m\choose b}$ is a bigraded map of bidegree $(0,n-b)$.
	In particular, when $b=0$, since $K_0^0=R$,  the map $K_0^0\otimes S^{m\choose n}\rightarrow S$ is induced by the $\epsilon$ map appearing in the Eagon-Northcott complex (\Cref{EagonNorthcott}).
	\item Using \cite[Lemma 2.6]{BKM1} the other/remaining horizontal maps in the rows are bigraded maps of bidegree $(0,1)$.
	In particular, since $K^0_a\cong (S_a)^*$, the maps in the last row of the bicomplex $\mathbb{B}_\bullet$ are induced by the maps in the Eagon-Northcott complex (\Cref{EagonNorthcott}).
\end{enumerate}
\end{remark}

The bigraded shifts of each free module appearing in the bicomplex $\mathbb B_\bullet$ are determined by the fact that each homomorphism has bidegree $(0,0)$. Hence, the domain of a bigraded map of bidegree $(a,b)$ is shifted by $(-a,-b)$. For ease of reference, we rewrite the bicomplex $\mathbb B_\bullet$ with only the bigraded shifts: 
\begin{align}
\xymatrix{
	0\ar[r] & (-(n-1),-m) \ar[r] \ar[d]& \cdots \ar[r]&  (-(n-1),-n)\ar[r]\ar[d] &(-(n-1),-(n-1))\ar[d]^{\phi^{n-1}}\\
	& \vdots &\vdots &\vdots& \vdots\\
	0\ar[r] & (-1,-m) \ar[r]\ar[d]& \cdots \ar[r]&(-1,-n)\ar[r]\ar[d] &(-1,-1)\ar[d]^{\phi^{1}}\\
	0\ar[r] & (0,-m) \ar[r]& \cdots \ar[r]&(0,-n)\ar[r] &(0,0).
}
\end{align}

Thus, we get the following:
\begin{prop}\label{Prop:shift}
	The bigraded shifts appearing in the $S$-free resolution \eqref{total complex} of the residual intersection $S/J$ 
	are as follows:
	\begin{align}
	Q_k&=S(-k,-k)^{\binom{m}{k}}, & 0\leq k\leq n-1;\\
	P_i&=\bigoplus_{j} S(-j,-(n+(i-1)-j))^{r(i,j)}, & 1\leq i\leq m, \label{formula for Pi}
	\end{align}
	where $\; \max\{0,i-(m-n+1)\} \; \leq \;  j \; \leq \; \min\{i-1,n-1\}$ and $$r(i,j)={n+i-2j-2\choose i-1-j}{n+i-1-j\choose j}{m\choose n+i-1-j}.$$
In particular, we have $\reg_x(S/J) = 0$, $\reg_y(S/J) = (n-1)$, and the bigraded Betti numbers of $S/J$ are given by 
\begin{align*}
\beta_{i,(a,b)}^S(S/J) =\begin{cases}
{m\choose i} & 0\leq i=a=b\leq n-1\\
r(i,a) & 1\leq i\leq m,  b=n+i-1-a, \\&\max\{0,i-(m-n+1)\} \; \leq \;  a \; \leq \; \min\{i-1,n-1\}\\
0 &\mathrm{\text{otherwise}}
\end{cases}
\end{align*}
\end{prop}

We conclude this section with the following remark, which is used  to discuss the Koszul, and Cohen-Macaulay properties, of $(S/J)_\Delta$ in the next two subsections.

\begin{remark}\label{rem:DepthRegularity}{\rm
The exactness of $(\_\!\_)_\Delta$ shows that 
\begin{align*}
 0\rightarrow (F_m)_{\Delta} \rightarrow \cdots\rightarrow (F_{i})_{\Delta}\rightarrow  \cdots \rightarrow (F_1)_{\Delta} \rightarrow S_{\Delta} \rightarrow (S/J)_\Delta \rightarrow 0
 \end{align*} is exact. 
 
 Thus, by \Cref{rem:regularity}(\ref{reglemma}), we see that $\depth_{S_{\Delta}} (S/J)_{\Delta} \geq \min \{ \depth_{S_{\Delta}} (F_i)_{\Delta} - i\ \vert\ 0\leq i\leq m \}$, and $\reg_{S_{\Delta}} (S/J)_{\Delta} \leq \max \{ \reg_{S_{\Delta}} (F_i)_{\Delta} -i\ \vert\ 0\leq i\leq m \}.$
}\end{remark}

\subsection*{Depth and the Cohen-Macaulay Property of $(S/J)_\Delta$}\label{subsec:CMDiagSubalg}\hfill{}\\
In this subsection, we retain the setup of the previous one, and identify lower bounds on the depth of $(S/J)_\Delta$, which gives sufficient conditions for it to be Cohen-Macaulay. 

\begin{theorem}\label{thm:DepthGenResInt}
	Suppose $p>m\geq n$, then $\depth (S/J)_\Delta\geq p+n-(m+1)$ for all $\Delta$.
\end{theorem}
\begin{proof} 
By \Cref{rem:regularity}(\ref{reglemma}) and \Cref{rem:DepthRegularity}, we have 
$$\depth_{S_{\Delta}} (S/J)_{\Delta} \geq \min \{ \depth_{S_{\Delta}} (F_i)_{\Delta} - i\ \vert\ 0\leq i\leq m \}.$$
We see that $\dim \left(  (F_i)_{\Delta} \right) = p+n-1$, since $\dim\left( S(-a,-b)_{\Delta}\right) = p+n-1$ for any pair $(-a,-b)$ by \Cref{remarkBigraded}(\ref{rem:S(-a,-b)deltaCM}). The bigraded shifts $(-a,-b)$, described in \Cref{Prop:shift}, satisfy $0 \leq a < n$, and $0 \leq b \leq m$. Hence they satisfy the conditions in \Cref{remarkBigraded}(\ref{rem:S(-a,-b)deltaCM}), and therefore, $S(-a,-b)_{\Delta}$ is Cohen-Macaulay for these shifts. This implies that $\left(  F_i \right)_{\Delta} $ is Cohen-Macaulay for all $i$. 

Thus, $\depth\left(  (F_i)_{\Delta} \right) = p+n-1$ for each $i$, and hence we see that $\depth (S/J)_{\Delta} \geq p+n-(m+1)$, proving the result.
\end{proof}

\begin{cor}\label{thm:CMGenResInt}
	Suppose $p>m\geq n$ and $\dim (S/J)_\Delta\leq p+n-(m+1)$, then $(S/J)_\Delta$ is Cohen-Macaulay.
\end{cor}

\subsection*{Regularity and the Koszul property of $(S/J)_\Delta$}\label{subsec:KoszulDiagSubalg}\hfill{}\\
The following is the main theorem of this subsection.
 \begin{theorem}\label{thm:KoszulGeomResInt}
 With notations as before, $(S/J)_{\Delta}$ is Koszul for all $c \geq 1$ and $e \geq \frac{n}{2} $.
 \end{theorem}

\begin{proof}
Note that $S_{\bigtriangleup}$ is Koszul by Remark \ref{rem:regularity}(\ref{thm:Froberg}). Using Remark \ref{rem:regularity}(\ref{thm:koszulcriterion}), it is enough to prove that $\reg_{S_{\bigtriangleup}} (S/J)_{\bigtriangleup} \leq 1$. Since $\reg_{S_{\Delta}} (S/J)_{\Delta} \leq \max \{ \reg_{S_{\Delta}} (F_i)_{\Delta} -i\ \vert\ 0\leq i\leq m \}$, we compute $\reg_{S_\Delta}((F_i)_\Delta)$ for all $i$. 

By Remark \ref{rem:regularity}(\ref{regofshifteddiagonal}), 
$$ \reg_{S_{\bigtriangleup}} (F_i)_{\bigtriangleup} = \max \left\{  {\small  \left\lceil  \frac{a_i^{\max}}{c}  \right\rceil,  \left\lceil  \frac{b_i^{\max}}{e}  \right\rceil } \right\},$$ where $a_i^{\max} = \max\{a\ \vert\ S(-a,-b) \text{ is a direct summand of } F_i\}$ and $b_i^{\max}$ is defined similarly.

Therefore we can write
$$ \reg_{S_{\bigtriangleup}} (S/J)_{\bigtriangleup} \leq  \max \left\{  {\small  \left\lceil  \frac{a^{\max}_i }{c}  \right\rceil } - i , {\small  \left\lceil  \frac{b^{\max}_i }{e}  \right\rceil } -i \;  \vert\ 0\leq i\leq m  \right\}.$$
By Proposition 2.3, notice that $a^{\max}_i=i$ for $i=0,1,\ldots, (n-2)$, and $a^{\max}_i=(n-1)$ for $i=(n-1),\ldots,m$. Therefore 
$$  \left\lceil  \frac{ a^{\max}_i }{c}  \right\rceil  -i  \leq \max \left\{ {\small   \left\lceil  \frac{i}{c}  \right\rceil } - i  \;  \vert\ 1 \leq i\leq (n-1)  \right\} \leq 1,  $$
for all $c \geq 1$. 

Similarly, when $e\geq \frac{n}{2}$, we can show that $\left\lceil  \frac{ b^{\max}_i }{e}  \right\rceil  -i \leq 1$ for all $i$. 
For example, when $m \leq 2n-1$, we have $b^{\max}_i= n-1 +i$ for $i=1,\ldots,(m-n)$ and $b^{\max}_i = m$ for $i= (m-n+1), \ldots, m$, and hence
 $$ \left\lceil  \frac{ b^{\max}_i }{e}  \right\rceil  -i \leq  \max \left\{ {\small   \left\lceil  \frac{n-1+i}{e}  \right\rceil } - i  \;  \vert\ 1 \leq i\leq (m-n+1)  \right\} =   \left\lceil \frac{n}{e} \right\rceil -1 \leq 1,$$ since $e \geq \frac{n}{2}$. 
The calculations when $m > 2n - 1$ are similar, proving the result.
 \end{proof}
\begin{cor}
	If $n=2$, then the diagonal subalgebra $(S/J)_\Delta$ is always Koszul.
\end{cor}


\section{Applications to perfect ideals of height two}\label{sec:Appl}
A natural source of geometric residual intersections are the Rees algebras of certain classes of perfect ideals of height two, which are linearly presented in  a polynomial ring. We use the following:

\begin{setup}\label{setup} Let $I$ be a homogeneous perfect ideal of height two in a polynomial ring $R_x = \sk[x_1,\dots,x_n]$ with a presentation matrix $\Phi$ (\Cref{ReesAlgebraDefn}). We assume that $I$ satisfies the following properties:
\begin{enumerate}
	\item $\mu(I)=p> n$.
	\item $I$ satisfies $\mu(I_\p)\leq \hgt \p$ for every $\p\in V(I)\backslash \{ \m\}$. Equivalently, $\hgt(I_{p-i}(\Phi))> i$ for $0\leq i<n$ (e.g., \cite[Proposition 20.6]{EisenbudCommAlg}).
	\item The presentation matrix $\Phi$ of $I$ is linear in the entries of $R_x$.
\end{enumerate}
\end{setup}
Since $I$ is a perfect ideal of height two, the presentation matrix $\Phi$ of $I$ is of size $p\times (p-1)$ by the Hilbert-Burch theorem (e.g., \cite[Theorem 20.15]{EisenbudCommAlg}). This also shows that the generators of $I$ are of degree $p-1$. Let $S=R_x[y_1,\dots,y_p]$ with bigrading $\deg x_i=(1,0)$, and $\deg y_j=(0,1)$. With notation as in \Cref{remarkReesAlgebras}(b), let $m = p-1$, and $[z_1\cdots z_{p-1}]=[y_1\cdots y_p]\ \Phi$. We can rewrite this as  $[z_1\cdots z_{p-1}]=[x_1\cdots x_n]\cdot \phi$ where $\phi$ is a $n\times (p-1)$ matrix with entries in $R_y$. Since $\Phi$ is linear in $\sk[x_1,\dots,x_n]$, we have $\phi$ is linear in $R_y$.

\begin{remark}\label{heightTwoperfectRemark}
	\rm{
With the notation as in \Cref{setup}, and $\phi$ as above, we get the following from \\\cite[Theorems 1.2, 1.3]{MorUlr} (and their proofs): The Rees algebra of $I$ can be written as   $\mathcal{R}(I)\cong S/J$, where  $J=\langle z_1,\dots,z_{p-1}\rangle+I_n(\phi)$. Furthermore, $\hgt (I_n(\phi))\geq m-n+1$, $J=\langle z_1,\dots,z_{p-1}\rangle:\langle \underline{x}\rangle$, and in particular, $J$ is a geometric $m$-residual intersection of $\langle \underline{x}\rangle$.  
}
\end{remark}

Thus \Cref{Prop:shift} gives a free resolution of $\mathcal R(I)$ over $S$. We can now apply the results from \Cref{Sec:DigGeoRes} to $\mathcal R(I)$.

\begin{theorem}\label{thm:ReesAlg}
With the notation as in \Cref{setup}, we have the following:
\begin{enumerate}[a)]
\item\label{prop: linearresolution} $I^s$ has a linear resolution for all $s \in \mathbb N$.
\item\label{thm:CM} $(\mathcal R(I))_{\Delta}$ is Cohen-Macaulay for all $\Delta = (c,e)$ with $c>(p-1)e$ or $c=e=1$.
\item $(\mathcal R(I))_{\Delta}$ is Koszul for all $\Delta = (c,e)$ with $c\geq 1$ and $ e \geq \frac{n}{2}$.
\end{enumerate}
\end{theorem}
\begin{proof}
a) This follows from \Cref{rem:regularity}(f), since $\reg_{x} \left( \mathcal{R}(I) \right) = 0$ by \Cref{Prop:shift}.\\
b) Recall that $m=p-1$ and,  by Remark \ref{remarkReesAlgebras} (\ref{rem:dimofdiagonalofreesalgebra}) and (\ref{rem:dimofdiagonalofreesalgebraForc=e=1}) , we have $\dim \mathcal{R}(I)_\Delta=n$ in both cases. Thus, by \Cref{thm:DepthGenResInt} and \Cref{thm:CMGenResInt}, we have $\depth (\mathcal{R}(I)_\Delta)=\dim (\mathcal{R}(I)_\Delta)=n$ for all $\Delta$.\\
c) This is follows immediately from \Cref{thm:KoszulGeomResInt}.
\end{proof}

\begin{remark}{\rm\hfill{}
\begin{enumerate}[a)] 
\item The special case of \Cref{thm:ReesAlg}(a) when $p=n+1$ is proved in \cite[Corollary 5.11]{Tim}.
\item In \cite[Corollary 3.15]{ConHerTruVal}, the authors show that the ${\mathcal{R}(I)}_\Delta$ is Cohen-Macaulay for ${{\Delta}}$ such that $c,e\gg0$. In Theorem \ref{thm:ReesAlg}(b), we give an explicit range for $c,e$. 
\item If $c>(p-1)e$, then ${\mathcal{R}(I)}_\Delta\cong \sk[(I^e)_c]$ by \cite[Lemma 1.2]{ConHerTruVal}. Thus by Theorem \ref{thm:ReesAlg}, $\sk[(I^e)_c]$ is Cohen-Macaulay. In \cite[Corollary 4.4]{CutHer}, the authors show that there exists a positive integer $f$ such that $\sk[(I^e)_c]$ is Cohen-Macaulay when $c\geq fe$. In Theorem \ref{thm:ReesAlg}(b), we give an explicit values of $f$.
\end{enumerate}
}\end{remark}

\section{Example}\label{Sec:Example}
Let $I$ be an in $\sk[x_1,x_2,x_3]$ whose presentation matrix is
\begin{align*}
\Phi=\begin{bmatrix}
x_1 & 0 & 0 &0\\
x_2 & x_1 & 0 & 0\\
x_3 & x_2 & x_1 & 0\\
0 & x_3 & x_2 & x_1\\
0 & 0 &x_3 & x_2
\end{bmatrix}.
\end{align*}
Notice that $\hgt(I_4(\Phi))\geq 2$ (as $x_1^4,x_2^4-3x_1x_2^2x_3+x_1^2x_3^2\in I_4(\Phi)$). By the Hilbert-Burch theorem (e.g., \cite[Theorem 20.15]{EisenbudCommAlg}), $I$ is a perfect ideal of height two in $\sk[x_1,x_2,x_3]$. One can easily check (2) of \Cref{setup} by noticing that $\hgt (I_4(\Phi))\geq 2, \hgt (I_3(\Phi))=3$. Since $\Phi$ is linear in $\sk[x_1,x_2,x_3]$, all the hypothesis of \Cref{setup} is satisfied. 

Now let $S=\sk[y_1,\dots,y_5,x_1,x_2,x_3]$ with bigrading $\deg y_i=(0,1)$ and $\deg x_i=(1,0)$. As mentioned in \Cref{heightTwoperfectRemark}, the Rees algebra $\mathcal{R}(I)\cong S/J$ where $J=\langle z_1,\dots,z_4\rangle+I_3(\phi)$ where $[z_1\cdots z_4]=[y_1\cdots y_5]\cdot \Phi=[x_1~x_2~x_3]\cdot\phi$ and $\phi$ can be obtained as follows
\begin{align*}
[z_1\cdots z_4]&=[y_1\cdots y_5]\cdot \Phi\\
&=\begin{bmatrix}
x_1y_1+x_2y_2+x_3y_3& x_1y_2+x_2y_3+x_3y_4& x_1y_3+x_2y_4+x_3y_5& x_1y_4+x_2y_5
\end{bmatrix}\\
&=[x_1~x_2~x_3]\begin{bmatrix}
y_1 & y_2 & y_3 & y_4\\
y_2 & y_3 &y_4 & y_5\\
y_3 & y_4 & y_5 & 0
\end{bmatrix}=[x_1~x_2~x_3]\cdot\phi.
\end{align*}
The discussion in \Cref{heightTwoperfectRemark} shows that $J$ is a geometric residual intersection with $\hgt \phi\geq 2$.
Using \Cref{Prop:shift} we can compute the complete resolution of $S/J$ as follows:
	 
\begin{multline*}
0\rightarrow S(-2,-4)^6\rightarrow S(-2,-3)^{12}\oplus S(-1,-4)^{8}\rightarrow S(-1,-3)^{12}\oplus S(0,-4)^3\oplus \\S(-2,-2)^6\rightarrow S(0,-3)^4\oplus S(-1,-1)^4\rightarrow S.
\end{multline*}
By \Cref{thm:ReesAlg}, we have that $\mathcal{R}(I)_\bigtriangleup\cong (S/J)_{\Delta} $ is Cohen-Macaulay for all $\Delta$ with $c>4e$ or $c=e=1$, and is Koszul for $ e \geq \frac{3}{2}$.

\end{document}